\documentclass[12pt]{amsart}

\usepackage{amssymb}
\usepackage{amsmath}
\usepackage{amsfonts}
\usepackage[utf8]{inputenc}
\usepackage{amsthm}
\usepackage{tikz-cd}
\usepackage{float}
\usepackage{diagbox}
\DeclareUnicodeCharacter{2212}{-}
\newtheorem{thm}{Theorem}[section]
\newtheorem{lemma}[thm]{Lemma}
\newtheorem{df}[thm]{Definition}
\newtheorem{prop}[thm]{Proposition}
\newtheorem{ex}[thm]{Example}
\newtheorem{cor}[thm]{Corollary}

\newtheorem{prob}[thm]{Problem}
\newtheorem{remark}[thm]{Remark}

\newcommand{\Z}{\mathbb{Z}}
\newcommand{\ord}{\mathrm{ord}}
\newcommand{\charr}{\mathrm{char}}
\newcommand{\R}{\mathbb{R}}
\newcommand{\Q}{\mathbb{Q}}
\newcommand{\N}{\mathbb{N}}
\newcommand{\C}{\mathbb{C}}
\newcommand{\OO}{\mathcal{O}}

\newcommand{\FF}{\mathbb{F}}

\newcommand{\bs}{\backslash}

\renewcommand{\thefootnote}{\fnsymbol{footnote}}

\title{On Waring numbers of henselian rings}
\author{Tomasz  Kowalczyk and Piotr Miska}
\date{}

\begin{document}

\thanks{The research of the second author was supported by the Polish
National Science Centre grant UMO-2018/29/N/ST1/00470.}
\maketitle

\begin{abstract}
    
Let $n>1$ be a positive integer.    Let $R$ be a henselian local ring with residue field $k$ of $n$th level  $s_n(k)$. We give some upper and lower bounds for the $n$th Waring  number $w_n(R)$ in terms of $w_n(k)$ and $s_n(k)$. In large number of cases we are able to compute $w_n(R)$. Similar results for the $n$th Waring number of the total ring of fractions of $R$ are obtained. We then provide applications. In particular we compute $w_n(\Z_p)$ and $w_n(\Q_p)$ for $n\in\{3,4,5\}$ and any prime $p$.

\end{abstract}
\section{Introduction}

For a ring $R$ and a positive integer $n$, we define its $n$th Waring number $w_n(R)$ as the smallest positive integer $g$ such that any sum of $n$th powers is a sum of at most $g$ $n$th powers. If no such number exists, we put $w_n(R)=\infty$. Of course $w_1(R)=1$ for any ring $R$. We define the $n$th level of $R$ as the smallest positive integer $g$ such that $-1$ can be written as a sum of $g$ $n$th powers and denote it by $s_n(R)$.  We put $s_n(R)=\infty$ if such number does not exist. Clearly, for odd $n$ we have $s_n(R)=1$.

In the literature, there is a notion of $2$-Pythagoras number, which in our language is just $w_2$. Also, higher Pythagoras numbers are considered, however they are usually considered only for even powers. Since in this paper we will consider both odd and even powers, we decided to adopt the name Waring numbers for the aforementioned invariants.

The first result about $2$nd Waring number of a ring was famous Lagrange Four Square Theorem, which in modern language states that the $2$nd Waring number of the ring of integers and the field of rationals is equal to $4$ i.e. $w_2(\mathbb{Z})=w_2(\Q)=4$.  In general, it is a very difficult task to determine whether $n$th Waring number is finite or not. There are also some results in this manner for the rings of $k$-regulous functions on $\R^n$ (see \cite{banecki}).

Let us now briefly describe what is known for $n=2$.
If $-1$ is not a sum of squares in $R$, i.e. $s_2(R)=\infty$, then $A$ is called a real ring. It was proved by Hoffman \cite{hoffman}, that any positive integer can be realized as a $2$nd Waring number of a real field. There are also a lot of results for real polynomial rings. In the paper \cite{CLDR}  it was shown that $w_2(\mathbb{Z}[x])=\infty$ and $w_2(K[x_1,x_2,\dots, x_s])=\infty$ provided that $s\geq 2$ and $K$ is a real field. Recently, some development for sums of squares of polynomials of low degree was made \cite{benoist}, see also \cite{benoist 2, blachut, fernando} for related results. The oldest question in this manner is to find a precise value of $w_2(\R(x_1,x_2,\dots, x_s))$. At this moment we only know that (cf. \cite{CLR})
$$s+2 \leq w_2(\R(x_1,x_2,\dots, x_s))\leq 2^s$$
for $s\geq 2$. It is known that the Motzkin polynomial $M(x,y)=x^2y^4+x^4y^2 -3x^2y^2+1$ is nonnegative on $\R^2$, but is not a sum of squares of polynomials. Furthermore, it is not a sum of three squares of rational functions \cite{Cassels}, hence $w_2(\R(x_1,x_2))=4$. Besides the above inequality, nothing more is known for $s\geq 3$.

Assume now that $-1$ is a sum of squares in $R$ i.e. $s_2(R)<\infty$. In this case, $R$ is called a nonreal ring. This instance is somehow simpler, since there are the following bounds on $w_2(R)$ (cf. \cite{CLDR}).
$$s_2(R)\leq w_2(A)\leq s_2(R)+2$$
If we further assume that $2$ is a unit in $R$, then
$s_2(R) \leq w_2(R) \leq s_2(R)+1$.
This inequality follows from the formula 
$$a= \left(\frac{a+1}{2}\right)^2-\left(\frac{a-1}{2}\right)^2,$$
where $a \in R$ is any element of $R$.
It was a remarkable result of \cite{DLP} that the nonreal ring
$$A_s=\R[x_1,x_2,\dots, x_s]/(x_1^2+x_2^2+\dots+x_s^2+1)$$
has $s_2(A_s)=s$. A computation of $2$nd Waring number of $A_s$ for $s$ of particular form was performed in \cite{CLDR}. Authors do not know, if the $2$nd Waring numbers of the remaining cases of $s$ are known. The situation is different if we assume that $A$ is a field. Pfister \cite{pfister 1}  showed that the $2$-level of a nonreal field is necessarily a power of $2$, and all the powers can occur. There are several generalizations of the notion of level (cf. \cite{hoffman leep}) and a survey \cite{hoffman 2} for open questions concerning higher levels and reference therein.

What if $n>2$? If $s_n(R)<\infty$, then we may use the following identity
$$n!x=\sum_{r=0}^{n-1}(-1)^{n-1-r} {n-1 \choose r} [(x+r)^n -r^n ]. $$
This easily shows that if $n!$ is an invertible element of $R$, then we get
$$w_n(R)\leq  nG(n)(s_n(R)+1). $$
Here $G(n)$ is the smallest positive integer $s$ such that any sufficiently large positive integer can be written as a sum of at most $s$ $n$th powers. In particular, it is known that $G(n)$ exists and is finite. This function is closely related to the classical Waring problem (see \cite{ellison, liu wooley, vaughan wooley, wooley 1, wooley 2}). See also \cite{alnaser, small, winterhof, winterhof 2} for Waring problem in finite rings, \cite{bhaskaran, siegel} for number fields, \cite{bhaskaran2, birch, bovey, joly3, moser1, moser2, moser3, voloch} for $p$-adic rings and \cite{ellison 2, joly1, joly2, joly4, niven, ramanujam, tornheim, vaserstein 1, vaserstein 2} for more general rings. There is an interesting algorithmic approach to the problem of sum of squares \cite{darkey1, darkey2, koprowski1, koprowski2}.
It is also worth to mention the paper \cite{vaughan wooley} and its enormous literature.

If $s_{2d}(R)=\infty$, the situation is much more difficult. There are not many results on this topic in the literature, however it is known that $w_{2d}(\Z[x])=\infty$ (see \cite{CLDR}) and $w_{2d}(\mathbb{R}[x,y])=\infty$ (see \cite{vill}). Also, some bounds are known for algebroid curves (see \cite{ruiz1, ruiz2, quarez}). It is worth to mention the following theorem proved by Becker (cf. \cite[Theorem 2.12]{becker}): if $R$ is a  field, then
$$w_2(R)<\infty \iff w_{2d}(R)<\infty$$
for every positive integer $d>0$. However, it is not known if this fact can be generalized to rings.

In this paper we tackle the problem of computing the values of the Waring numbers and levels of a henselian local ring $R$.
Under some mild assumptions on $R$ we are able to prove upper bounds on $w_n(R)$. Also, in a quite large class of local henselian rings, we show this upper bound is achieved. Based on this, we study the $n$th Waring number of the total ring of fractions or $R$. The picture is much more complete, if we further assume that $R$ is a valuation ring of a rank-1 valuation. As we will see, in general computing the Waring number and level of henselian local ring $R$ and its full ring of quotients comes down to computing the Waring number and level of some quotient ring of $R$ (in many cases this quotient ring is the residue field $k$ of $R$).

The structure of the paper is as follows. Section 2 contains preliminary results concerning local henselian rings and Waring numbers. We then proceed to discuss henselian local rings in the Section 3, as well as henselian valuation rings of rank-1. The main result of this section is Theorem \ref{pnR} which gives us a method for computing $w_n(R)$ under some mild assumptions on $R$. Section 4 consists of results concerning total rings of fractions of a local henselian rings. Section 5 concerns Waring numbers for henselian rings with infinite level and their quotient fields. Section 6 contains numerous applications of theory developed in previous sections. In particular, we consider Waring numbers of various power series rings and fields. Also some lower bounds are given in the case of a coordinate ring of an irreducible algebraic subset. We then provide application to the $p$-adic rings and fields, and give full classification of $w_n(\Z_p)$ and $w_n(\Q_p)$ for $n=3,4,5$. The last subsection of Section 6 contains a discussion of possible relations between Waring numbers of a local ring, its henselization, and its completion. We finish the paper with some open problems which could stimulate further research in this manner.

\section{Preliminaries.}

We begin with a few definitions. Let $R$ be a ring and $n>1$ be a positive integer. For an element $a \in R$ we define \emph{the n-length of $a$}, denoted by $\ell_n(a)$ as the smallest positive integer $g$, such that $a$ can be written as a sum of $g$ $n$th powers. We put $\ell_n(a)=\infty$ if such number does not exist. Note that the $n$th level of $R$ is precisely the $n$-length of $-1$, i.e. $\ell_n(-1)=s_n(R)$.

Let us also mention the following property of the $n$th Waring number.
If $S$ is a multiplicatively closed set, then $w_n(S^{-1}R)\leq w_n(R)$. In particular, if $R$ is an integral domain with a field of fractions $K$, then $w_n(K)\leq w_n(R)$.

Let $R$ be an integral domain and $K$ be its field of fractions. If $R$ is local, then we denote its maximal ideal by $\mathfrak{m}$. The quotient ring $k=R/\mathfrak{m}$ is called the residue field of $R$. For an element $a \in R$, we denote its reduction modulo $\mathfrak{m}$ by $\Bar{a}$. 

A rank-$1$ valuation on a ring $R$ is a map $\nu : R \rightarrow \R \cup \{\infty \}$ such that
\begin{itemize}
    \item[a)] $\nu(x)=\infty \iff x=0$
    \item[b)] $\nu(xy)=\nu(x)+\nu(y)$
    \item[c)] $\nu(x+y) \geq \min \{\nu(x), \nu(y) \}$
\end{itemize}
for all $x,y \in R$. If $\nu(R\bs\{0\})$ is a discrete subset of $\R$ with respect to the Euclidean topology, we then say that $\nu$ is discrete. Let us note that existence of a rank-$1$ valuation on $R$ forces that $R$ is integral. We thus may extend a valuation to the field $K$ of fractions of $R$ by formula $\nu\left(\frac{x}{y}\right)=\nu(x)-\nu(y)$, $x,y\in R$, $y\neq 0$. If $R=\{f\in K|\nu(f)\geq 0\}$, then we will say that $R$ is a rank-$1$ valuation ring. The unique maximal ideal of $R$ is of the form $\mathfrak{m}=\{x \in R | \nu(x)>0 \}$.

If $\nu$ is rank-$1$ valuation on a field $K$, then the set $R_\nu=\{f\in K|\nu(f)\geq 0\}$ is a rank-$1$ valuation ring. The set $\nu(K\bs\{0\})$ is an Abelian group and is called valuation group of $\nu$. If moreover $\nu$ is discrete, then $\nu(K\bs\{0\})$ is of the form $\alpha\Z$ for some $\alpha>0$. After dividing $\nu$ by $\alpha$ we may assume that $\nu(K\bs\{0\})=\Z$. Hence, from now on, we assume that every discrete valuation attains only integer values. If $\nu$ is a discrete valuation on $K$, then  $R_\nu$ is a \emph{discrete valuation ring} or \emph{DVR}, for short. Each DVR is a local principal ideal domain and vice versa. A generator of $\mathfrak{m}$ is called then the uniformizing element or uniformizer of $R_\nu$.

We say that a local ring $R$ is \emph{henselian}, if for every polynomial $f \in R[x]$ in a single variable $x$, if there exists $b \in R$ such that $f(b)\in\mathfrak{m}$ and $f'(b)\not\in\mathfrak{m}$, then there exists an element $a \in R$ such that $f(a)=0$ and $\Bar{a}=\Bar{b}$. If $R$ is a rank-$1$ valuation ring with valuation $\nu$, then $R$ is henselian if and only if for every $f \in R[x]$, if there exists $b \in R$ such that $\nu(f(b))>2\nu(f'(b))$, then there exists an element $a \in R$ such that $f(a)=0$ and $\nu(a-b)>\nu(f'(b))$ (see \cite[Theorem 4.1.3]{EngPres}).

We call a ring $R$ reduced if it has no nontrivial nilpotents.

Before we prove the general results, we need some lemmas.
\begin{lemma}\label{surjection}
Let $\varphi : R \rightarrow S$ be a homomorphism of rings. Then for any $x\in R$ and for any positive integer $n>1$ the following inequality holds $\ell_n(\varphi(x)) \leq \ell_n(x)$. If $\varphi$ is an epimorphism, then $w_n(S)\leq w_n(R)$.
\begin{proof}
Clear.
\end{proof}
\end{lemma}

\begin{df}\label{free part}
Let $p$ be a prime integer and $n$ be a positive integer written in the form $n=p^km$, where $p$ does not divide $m$ and $k\geq 0$. We then say that $m$ is the $p$-free part of $n$. We extend this definition to the case $p=0$ and put $n=m$.
\end{df}

\begin{prop}\label{p_m}
Let $R$ be a ring of prime characteristic $p$ and $n>1$ be a positive integer.  If $n=p^km$, where $m$ is the $p$-free part of $n$, then 
$$w_n(R)\leq w_m(R).$$
If we further assume that $R$ is reduced, then
$$w_n(R)=w_m(R).$$

\begin{proof}
Let $f \in R$ be a sum of $n$th powers, $f=\sum_{i=1}^kf_i^n$. We then have
$$f=\sum_{i=1}^kf_i^n=\sum_{i=1}^kf_i^{p^km}=\sum_{i=1}^k(f_i^m)^{p^k}=\left(\sum_{i=1}^kf_i^m\right)^{p^k}.$$
and
$$\ell_n(f) = \min_g \{\ell_m(g) : g^{{p^k}}=f \}. $$
Hence, the first part follows. If $R$ is reduced, then the Frobenius map $R \ni x \mapsto x^p \in R $ is injective. Hence, there exists exactly one element $g \in R$ such that $g^{p^k}=f$.

\end{proof}

\end{prop}

Injectivity of the Frobenius map is equivalent to nonexistence of nontrivial nilpotent elements. The assumption that $R$ is reduced cannot be omitted in the above proposition as the following example shows.
\begin{ex}\label{frob} {\rm
Let $\FF_3$ be a field with $3$ elements. Define the ring
$$R:=\FF_3[x]/(x^3)$$ and take $n=6$. Here, the $3$-free part of $6$ is equal to $2$.

Any sixth power in this ring is either $0$ or $1$, hence $w_6(R)=2$ as $2$ is not a sixth power. However, $x$ is a not sum of two squares since no sum of two nonzero squares of elements in $\FF_3$ is equal to zero, hence $\ell_2(x)>2$. It is precisely $3$ as $x=(1+x)^2+(1+x)^2+(1+2x)^2$. In particular $w_2(R)=3$ and
$$2=w_6(R)<w_2(R)=3.$$  

Note that $R$ is complete with respect to its maximal ideal and thus henselian. }
\end{ex}

\begin{remark}\label{s_m} { \rm
Analogously as Proposition \ref{p_m} one can prove that\linebreak $s_n(R)\leq s_m(R)$. On the other hand, since $m\mid n$, we have that every presentation of $-1$ as a sum of $n$th powers is a presentation as a sum of $m$th powers as well. Thus, $s_n(R)\geq s_m(R)$. Finally, we get that if $m$ is the $\charr(R)$-free part of $n$, then $s_n(R)=s_m(R)$. }
\end{remark}

\section{Waring numbers of henselian rings with finite $n$th level.}

Now, we are ready to prove the first results of the paper.

\begin{thm}\label{pnR}
Let $R$ be a local ring with the maximal ideal $\mathfrak{m}$ and the residue field $k$. Let $n$ be a positive integer and $m$ be the $\charr(k)$-free part of $n$. Assume that $\charr(k)\nmid n$ or $\charr(R)=\charr(k)$ and $R$ is reduced. Then, the following statements are true.
\begin{enumerate}
    \item[a)] We have $s_n(R)\geq s_m(k)$ and $w_n(R)\geq w_m(k)$.
    \item[b)] If $R$ is henselian and $s_m(k)<\infty$, then $s_n(R)=s_m(k)$ and $w_n(R)\leq\max\{w_m(k),s_m(k)+1\}$.
    \item[c)] If $f\in\mathfrak{m}\neq\mathfrak{m}^2$ and $m>1$, then $\ell_n(f)\geq s_m(k)+1$. In particular, $w_n(R)\geq s_m(k)+1$.
\end{enumerate}
\end{thm}

\begin{proof}
By Proposition \ref{p_m}, we have $s_n(R)=s_m(R)$ and $w_n(R)=w_m(R)$ in case of $\charr(R)=\charr(k)$ on condition that $R$ is reduced. Hence, for the rest of the proof, we may assume that $n=m>1$.

a) follows directly from Lemma \ref{surjection}.

We start the proof of b) with taking any $f\in R\bs\mathfrak{m}$ such that $\ell_n(\Bar{f})=l$. Consider the polynomial
$$p(x_1,x_2,\dots,x_l)=\sum_{i=1}^lx_i^n -f$$
in $R[x_1,x_2,\dots,x_l]$. The derivative of $p$ with respect to $x_i$ is equal to $p_{x_i}'=nx^{n-1}_i$. Let $g_i \in R$, $i\in\{1,\ldots,l\}$, be such that $\Bar{f}=\sum_{i=1}^l \Bar{g}_i^n$. Since $f\not\in\mathfrak{m}$, we have $\Bar{g}_{i_0}\neq 0$ in $k$ for some $i_0\in\{1,\ldots,l\}$, or equivalently, $g_{i_0}\not\in\mathfrak{m}$. Without loss of generality we may assume that $i_0=1$.
Since $R$ is henselian, there exists an element $h$ such that $h\equiv g_1 \pmod{\mathfrak{m}}$ and $p(h,g_2,\dots,g_l)=0$.
Hence, $f$ can be written as a sum of $l$ $n$th powers of elements of $R$. In particular, $s_n(R)=s_n(k)$.

If $f\in\mathfrak{m}$, then we may find a relation in the residue field $0=1+\sum_{i=1}^{s_n(k)1}\Bar{g}_i^n$ with $\Bar{g}_i \in k$. Again, as above, by henselianity of $R$ we conclude that $f$ can be written as a sum of $s_n(k)+1$ $n$th powers of elements of $R$. 

For the proof of c), consider an element $f\in\mathfrak{m}\bs\mathfrak{m}^2$. We will show that $\ell_n(f)\geq s_n(k)+1$. Assume by contrary that $f = \sum_{i=1}^l f_i^n$ for some $f_i \in R$ and $l<s_n(k)+1$. If $f_{i_0}\not\in\mathfrak{m}$ for some $i_0\in\{1,\ldots,l\}$, then the following holds in the residue field.
$$-1=\sum_{\substack{1\leq i\leq l\\ i\neq i_0}} \overline{\left(\frac{f_i}{f_{i_0}}\right)}^n$$
This is a contradiction with the definition of $s_n(k)$. As a consequence, $f_i\in\mathfrak{m}$ for all $i\in\{1,\ldots,l\}$ but then $f=\sum_{i=1}^l f_i^n\in\mathfrak{m}^n\subset\mathfrak{m}^2$, a contradiction.

In order to show $w_n(R)\geq s_m(k)+1$, let us consider the $n$-length of $f\in\mathfrak{m}\bs\mathfrak{m}^2$ in the henselization $R^h$ of $R$. From the proof of b) for $R^h$ in the place of $R$ we know that this length is at most equal to $s_m(k)+1$. Write $f = \sum_{i=1}^l f_i^n$ for some $f_i \in R^h$ and $l\leq s_n(k)+1$. For each $i\in\{1,\ldots,l\}$ pick $\tilde{f}_i\in R$ such that $\tilde{f}_i\equiv f_i\pmod{\mathfrak{m}^2}$ and put $\tilde{f} = \sum_{i=1}^l \tilde{f}_i^n$. Then $\tilde{f}\in\mathfrak{m}\bs\mathfrak{m}^2$ and the $n$-length of $\tilde{f}$ in $R$ is at most equal to $s_n(k)+1$, in particular finite. On the other hand we already know that this $n$-length is at least equal to $s_n(k)+1$. This allows us to state that $w_n(R)\geq s_n(k)+1$.
\end{proof}

\begin{cor}
With the assumptions as in Theorem \ref{pnR}, if $\charr(k)\nmid n$ and every element of $k$ can be written as a sum of $n$th powers in $k$, then every element of $R$ can be written as a sum of $w_n(R)$ $n$th powers in $R$.

\end{cor}

\begin{cor}\label{pnR=...}
Let $R$ be a henselian local ring with the maximal ideal $\mathfrak{m}\neq\mathfrak{m}^2$, residue field $k$ and $s_m(k)<\infty$. Let $n$ be a positive integer and $m$ be the $\charr(k)$-free part of $n$. Assume that $\charr(k)\nmid n$ or $\charr(R)=\charr(k)$ and $R$ is reduced. Then the following holds
\begin{equation}\label{eq3}
            w_n(R)=\begin{cases}
            \max \{w_m(k), s_m(k)+1 \} & \text{ for } m>1    \\
            1 & \text{ for } m=1
            \end{cases}.
        \end{equation}
\end{cor}

\begin{remark} {\rm
If $\charr(R)=\charr(k)\mid n$, $m>1$ and $R$ is not reduced, then Proposition \ref{p_m} and the reasoning in the above proof allows us to state that $w_n(R)\leq\max\{w_m(k),s_m(k)+1\}$ still holds in the point c). Example \ref{frob} shows that the inequality can be strict. }

\end{remark}

\begin{remark}\label{almostPuiseaux} {\rm
The assumption $\mathfrak{m}\neq\mathfrak{m}^2$ cannot be omitted in the point c). Consider the ring $R=\bigcup_{j=0}^\infty \C[[x^{\frac{1}{2^j}}]]$. Then $\mathfrak{m}=\bigcup_{j=0}^\infty x^{\frac{1}{2^j}}\C[[x^{\frac{1}{2^j}}]]=\mathfrak{m}^2$ and $k=\C$. Of course, $s_n(\C)=1$ for each positive integer $n$. Each element of $R$ can be written in the form $x^{\frac{a}{2^j}}f_0$ for some $a,j\in\N$, where $f_0\in\C[[x^{\frac{1}{2^j}}]]$ has order $0$ with respect to $x^{\frac{1}{2^j}}$. For each $t\in\N_+$ by Hensel's lemma we have $f_0=f_t^{2^t}$ for some $f_t\in\C[[x^{\frac{1}{2^j}}]]$. Thus, $$f=\left(x^{\frac{a}{2^{j+t}}}f_t\right)^{2^t}$$ and because of arbitrariness of $f$ we conclude that $w_{2^t}(R)=1<2=s_{2^t}(k)+1$ for any $t\in\N$. On the other hand, $x$ is not an $n$th power in $R$ when $n$ is not a power of $2$. Hence, $w_n(R)\geq 2=s_n(k)+1$ for any positive integer $n$ not being a power of $2$.

Summing up, if $\mathfrak{m}=\mathfrak{m}^2$, then the inequality from the point c) may hold but need not to. }
\end{remark}
It is worth to mention the paper \cite{joly4}. Among many other problems, Joly considers the finiteness of $w_d(A)$ for a henselian local ring $A$.

\begin{thm}\label{joly}\cite[Th\'eor\`eme 7.34]{joly4}
Let $A$ be a henselian local ring with a finite residue field and $d$ be a prime integer. Then there are upper bounds of the form
$$w_d(A)\leq 
\begin{cases}
6 \,\, \textit{if} \,\, d=2 \\
2d-1 \,\, \textit{if}\,\, d>2.
\end{cases}$$
\end{thm}
 The above Theorem works well for $d=p$ and any henselian local ring with finite residue field. In particular, it shows that for rings of the form $\Z_p[[x]]$, the $p$th Waring number is finite. The authors are not able to establish such a statement using theory developed in this paper (cf. Remark \ref{remark problems}).

\subsection{\textbf{\textsc{Waring numbers of henselian rank-$1$ valuation rings.}}}
In this subsection we deal with the problem of computing Waring numbers of rank-$1$ valuation rings. Assume for the remainder of this section, that $R$ is a rank-$1$ valuation ring with respect to a nontrivial valuation $\nu$. Denote by $K$ the field of fractions of $R$. Let $\mathfrak{m}$ be the maximal ideal of $R$. Define $I_r:=\{f\in R|\nu(f)>r\}$ (with this notation $\mathfrak{m}=I_0$). Let $n$ be a positive integer and $m$ be the $\charr(R)$-free part of $n$.

Performing analogous reasoning as in the proof of Theorem \ref{pnR} we can show the following result for the rings with rank-$1$ valuation.

\begin{thm}\label{pncvR}
With the above assumptions, the following statements are true.
\begin{enumerate}
    \item[a)] We have $s_n(R)\geq s_m(R/I_r)$ and $w_n(R)\geq w_m(R/I_r)$ for each $r\in [0,\infty)$.
    \item[b)] If $m>1$ and $r\in [0,\infty)$, $g\in K$ are such that\linebreak $\inf\{n\nu(f)|f\in\mathfrak{m}\}>\nu(g)>r$, then $\ell_n\left(g^{\frac{n}{m}}\right)\geq s_m(R/I_{r})+1$.
    \item[c)] If $R$ is henselian, $r\geq 2\nu(m)$ and $s_m(R/I_{r})<\infty$, then $s_n(R)=s_m(R/I_{r})$ and 
    \begin{enumerate}
        \item[i)] $w_n(R)=\max\{w_m(R/I_{r}),s_m(R/I_{r})+1\}$ if $m>1$ and\linebreak $\inf\{n\nu(f)|f\in\mathfrak{m}\}>\nu(g)>r$ for some $g\in\mathfrak{m}$;
        \item[ii)] $w_n(R)=w_m(R/I_{r})$ if $m>1$ and there exists $f\in\mathfrak{m}$ such that $\nu(g)\in [0,r]\cup [n\nu(f),\infty]$ for all $g\in R$;
        \item[iii)] $w_n(R)=1$ if $m=1$.
    \end{enumerate}
    Moreover, if $\charr(R)\nmid n$ and every element of $R/I_{r}$ can be written as a sum of $n$th powers in $R/I_{r}$, then every element of $R$ can be written as a sum of $w_n(R)$ $n$th powers in $R$.
\end{enumerate}
\end{thm}

\begin{proof}
If $n=m$, i.e. $n$ is not divisible by $\charr (R)$, then the theorem gives us the inequalities for $w_n(R)$. In general, by Proposition \ref{p_m}, we have $w_n(R)= w_m(R)$ in case of $\charr(R)\mid n$. Moreover, from the proof of Proposition \ref{p_m} we have $\ell_n\left(g^{\frac{n}{m}}\right)=\ell_m(g)$. Furthermore, it is obvious that $w_1(R)=s_1(R)=1$. Hence, for the rest of the proof, we may assume that $n=m>1$.

a) follows directly from Lemma \ref{surjection}.

We start the proof of b). Assume by contrary that $g = \sum_{i=1}^l g_i^n$ for some $g_i \in R$ and $l<s_n(R/I_r)+1$. If $\nu(g_{i_0})=0$ for some $i_0\in\{1,\ldots,l\}$, then the following holds after reduction modulo $I_r$.
$$-1=\sum_{\substack{1\leq i\leq l\\ i\neq i_0}} \overline{\left(\frac{g_i}{g_{i_0}}\right)}^n$$
This is a contradiction with the definition of $s_n(R/I_r)$. As a consequence, $g_i\in\mathfrak{m}$ for all $i\in\{1,\ldots,l\}$ but then $\nu(g)=\nu\left(\sum_{i=1}^l g_i^n\right)\geq\inf\{n\nu(f)|f\in\mathfrak{m}\}$, a contradiction. Thus, $\ell_n(g)\geq s_n(R/I_r)+1$.

We start the proof of c) with taking any $g\in R\bs I_r$ such that $\nu(g)<\inf\{n\nu(f)|f\in\mathfrak{m}\}$ and $\ell_n(\Bar{g})=l$ with respect to the quotient ring $R/I_{r}$. Consider the polynomial
$$p(x_1,x_2,\dots,x_l)=\sum_{i=1}^lx_i^n -g$$
in $R[x_1,x_2,\dots,x_l]$. The derivative of $p$ with respect to $x_i$ is equal to $p_{x_i}'(x_1,x_2,\dots,x_l)=nx^{n-1}_i$. Let $g_i \in R$, $i\in\{1,\ldots,l\}$, be such that $\Bar{g}=\sum_{i=1}^l \Bar{g}_i^n$ in $R/I_{r}$. Since $\nu(g)<\inf\{n\nu(f)|f\in\mathfrak{m}\}$, we have $\nu(g_{i_0})=0$ for some $i_0\in\{1,\ldots,l\}$. Without loss of generality we may assume that $i_0=1$.
Since $R$ is henselian and $\nu\left(p(g_1,g_2,\dots,g_l)\right)>r\geq 2\nu(n)=2\nu(p_{x_1}'(g_1,g_2,\dots,g_l))$, there exists an element $h\in R$ such that $h\equiv g_1 \pmod{I_{\nu(n)}}$ and $p(h,g_2,\dots,g_l)=0$.
Hence, $g$ can be written as a sum of $l$ $n$th powers of elements of $R$. In particular, $s_n(R)=s_n(R/I_r)$.

If $r<\nu(g)$, then we may find a relation in the quotient ring $R/I_{r}$ $$0=1+\sum_{i=1}^{s_n(R/I_r)}\Bar{g}_i^n$$ with elements $\Bar{g}_i \in R/I_{r}$ not all equal to $0$. Again, as above, by henselianity of $R$ we conclude that $g$ can be written as a sum of $s_n(R/I_r)+1$ $n$th powers of elements of $R$.

Assume now that there exists $f\in\mathfrak{m}$ such that $\nu(g)\in [0,r]\cup [n\nu(f),\infty]$. If $\nu(g)\geq n\nu(f)$, then (as $\nu(f)>0$) there exists $t\in\N$ such that $0\leq\nu\left(\frac{g}{f^{tn}}\right)=\nu(g)-tn\nu(f)<n\nu(f)$ (in fact, $t=\left\lfloor\frac{\nu(g)}{n\nu(f)}\right\rfloor$). By our assumption we have that $\nu\left(\frac{g}{f^{tn}}\right)\leq r$, so $\frac{g}{f^{tn}}\not\in I_r$ and we already know that $\ell_n\left(\frac{g}{f^{tn}}\right)\leq w_n(R/I_r)$. Thus $\ell_n(g)\leq\ell_n\left(\frac{g}{f^{tn}}\right)\leq w_n(R/I_r)$.

Summing up, if $\inf\{n\nu(f)|f\in\mathfrak{m}\}>\nu(g)>r$, then $w_n(R)\leq\max\{p_m(R/I_{r}),s_m(R/I_{r})+1\}$ and the inequalities $w_n(R)\geq w_m(R/I_{r})$ and $w_n(R)\geq s_m(R/I_{r})+1$ follow from points a) and b), respectively. If there exists $f\in\mathfrak{m}$ such that $\nu(g)\in [0,r]\cup [n\nu(f),\infty]$ for all $g\in R$, then $w_n(R)\leq w_m(R/I_{r})$ and the reverse inequality follows from a).
\end{proof}

\begin{remark} { \rm
Let us see that if $\charr(R)\mid n$, then $\nu(m)=0$. This is because $R$ is integral so $\charr(R)$ is a prime number. Thus, $\charr(R/\mathfrak{m})=\charr(R)$, which means that $\charr(R/\mathfrak{m})\nmid m$. Consequently, $m\neq 0$ in $R/\mathfrak{m}$. In other words, $m\not\in\mathfrak{m}$.

Hence, if $\nu(m)>0$, then $m=n$. In this case, $\charr(R)\nmid n$ but $\charr(R/\mathfrak{m})\mid n$. Since $R/\mathfrak{m}$ is a field, $R$ is integral, $\charr(R/\mathfrak{m})\mid\charr(R)$ and $\charr(R/\mathfrak{m})\neq\charr(R)$, we have $\charr(R)=0$ and $\charr(R/\mathfrak{m})$ is a prime number. }
\end{remark}

\begin{remark} { \rm
The case c-i) holds only if $\nu$ is discrete. Indeed, there exists $g\in\mathfrak{m}$ such that $\nu(f)>\frac{\nu(g)}{n}>0$ for all $f\in\mathfrak{m}$. Since $0$ is an isolated point of the valuation group of $\nu$, this group is discrete. }
\end{remark}

\begin{cor}\label{pnHDVR}
Assume that $R$ is a henselian DVR with the valuation group equal to $\Z$. Let $s_m(R/\mathfrak{m}^{2\nu(m)+1})<\infty$. Then $s_n(R)=s_m(R/\mathfrak{m}^{2\nu(m)+1})$ and 
\begin{enumerate}
        \item[i)] $w_n(R)=\max\{w_m(R/\mathfrak{m}^{2\nu(m)+1}),s_m(R/\mathfrak{m}^{2\nu(m)+1})+1\}$ if $m>1$ and $n>2\nu(m)+1$;
        \item[ii)] $w_n(R)=w_m(R/\mathfrak{m}^{2\nu(m)+1})$ if $m>1$ and $n\leq 2\nu(m)+1$;
        \item[iii)] $w_n(R)=1$ if $m=1$.
    \end{enumerate}
    Moreover, if $\charr(R)\nmid n$ and every element of $R/\mathfrak{m}^{2\nu(m)+1}$ can be written as a sum of $n$th powers in $R/\mathfrak{m}^{2\nu(m)+1}$, then every element of $R$ can be written as a sum of $w_n(R)$ $n$th powers in $R$.
\end{cor}

\begin{proof}
The case $m=1$ is obvious, so assume that $m>1$. We put $r=2\nu(m)$. If $n>2\nu(m)+1$, then we have $\inf\{n\nu(f)|f\in\mathfrak{m}\}=n>2\nu(m)+1>2\nu(m)=r$, so we apply Theorem \ref{pncvR} c-i). If $n\leq 2\nu(m)+1$, then the valuation of any element of $R$ belongs to $[0,2\nu(m)]$ or $[n,\infty]$, so we use Theorem \ref{pncvR} c-ii).
\end{proof}

\section{Waring numbers of total rings of fractions of\linebreak henselian rings with finite $n$th level}

\begin{thm}\label{pnQR}
Let $R$ be a henselian local ring with the total ring of fractions $Q(R)\neq R$, the maximal ideal $\mathfrak{m}$ and the residue field $k$. Let $n$ be a positive integer and $m$ be the $\charr(k)$-free part of $n$. Assume that $\charr(k)\nmid n$ or $\charr(R)=\charr(k)$. Then,
\begin{equation}\label{eq4}
    w_n(Q(R))\leq\begin{cases}
    s_m(k)+1 & \text{ for } m>1\\
    1 & \text{ for } m=1
    \end{cases},
\end{equation}
where the equality in the case of $m>1$ holds under assumption that $R$ is an integral domain, $s_m(Q(R))=s_m(R)$ and there exists a nontrivial valuation $\nu:Q(R)\to\Z\cup\{\infty\}$ such that $R\subset R_\nu:=\{f\in Q(R)\left|\nu(f)\geq 0\right.\}$. Moreover, if $\charr(k)\nmid n$, then every element of $Q(R)$ can be written as a sum of $w_n(Q(R))$ $n$th powers in $Q(R)$.
\end{thm}

\begin{proof}

By Proposition \ref{p_m}, we have $w_n(Q(R))\leq w_m(Q(R))$ in case of $\charr(R)=\charr(k)$ and the equality holds on condition that $R$ is reduced or $m=1$. Furthermore, it is obvious that $w_1(R)=w_1(Q(R))=1$. Hence, for the rest of the proof, we may assume that $n=m>1$.

Take any $f \in Q(R)$. Since we are dealing with a total ring of fractions being a proper superset of $R$, it is enough to find $g,f_1,\ldots,f_{l} \in R$ such that $g^nf=\sum_{i=1}^{l}f_i^n\in\mathfrak{m}$, where $g$ is not a zero divisor. By the proof of Theorem \ref{pnR} b) we know that $l=s_n(k)+1$ suffices. This proves the inequality $w_n(Q(R))\leq s_n(k)+1$. 

For the proof of the reverse inequality we make an assumption that $R$ is integral, $s_m(Q(R))=s_m(R)$ and $R\subset R_\nu=\{f\in Q(R)\left|\nu(f)\geq 0\right.\}$ for some nontrivial discrete valuation $\nu$ on the field $Q(R)$ of fractions of $R$. Since $\nu(Q(R))\bs\{0\}$ is a nontrivial subgroup of $\Z$, we may assume without loss of generality that $\nu(Q(R))=\Z\cup\{\infty\}$. Then, the set $\nu(R\bs\{0\})$ must be a submonoid of $\N$ generated by elements with their greatest common divisor equal to $1$. Otherwise $\nu(Q(R))\bs\{0\}$ would be a proper subgroup of $\Z$. Thus, there exists a positive integer $t$ not divisible by $n$ which is attained as a valuation of some element of $R$. Let us take $f\in R$ with $\nu(f)=t$. Suppose by contrary that $f = \sum_{i=1}^l f_i^n$ for some $f_i \in Q(R)$ and $l<s_n(k)+1$. If $\nu(f_i)>\frac{t}{n}$ for all $i\in\{1,\ldots,l\}$, then $\nu\left(\sum_{i=1}^l f_i^n\right)>t$ - a contradiction. Therefore, there exists $i_0\in\{1,\ldots,l\}$ such that $\nu(f_{i_0})=\min\{\nu(f_i)|i\in\{1,\ldots,l\}\}<\frac{t}{n}$. Dividing by $f_{i_0}^n$ we get $$\frac{f}{f_{i_0}^n}=1+\sum_{\substack{1\leq i\leq l\\ i\neq i_0}} \left(\frac{f_i}{f_{i_0}}\right)^n.$$ After reduction modulo $\mathfrak{m}_\nu=\{f\in Q(R)\left|\nu(f)>0\right.\}$ we obtain the following equality in the residue field $k_\nu=R_\nu/\mathfrak{m}_\nu$ of the ring $R_\nu$. $$-1=\sum_{\substack{1\leq i\leq l\\ i\neq i_0}} \overline{\left(\frac{f_i}{f_{i_0}}\right)}^n$$ This means that $s_n(k_\nu)\leq l-1<s_n(k)$. On the other hand, from Theorem \ref{pnR} b) we know that $s_n(R)=s_n(k)$ and $s_n(R_\nu)=s_n(k_\nu)$. Since $R\subset R_\nu\subset Q(R)$, we have $s_n(R)\geq s_n(R_\nu)\geq s_n(Q(R))$. Because we assume that $s_n(Q(R))=s_n(R)$ we infer that $s_n(k_\nu)=s_n(k)$ - a contradiction. Summing up, we showed that $w_n(Q(R))\geq s_n(k)+1$ if $R$ is an integral ring contained in the set $\{f\in Q(R)\left|\nu(f)\geq 0\right.\}$ for some nontrivial discrete valuation $\nu$ on $Q(R)$ and $s_n(Q(R))=s_n(R)$.
\end{proof}

\begin{remark} {\rm
If we omit the assumption that\linebreak $R\subset\{f\in Q(R)\left|\nu(f)\geq 0\right.\}$ for some nontrivial discrete valuation $\nu$ on $Q(R)$, then the equality $w_n(Q(R))= s_m(k)+1$, $m>1$, may hold but do not need to. Take $R=\bigcup_{j=0}^\infty \C[[x^{\frac{1}{2^j}}]]$. Then $Q(R)=K=\bigcup_{j=0}^\infty \C((x^{\frac{1}{2^j}}))$ and $k=\C$. By Remark \ref{almostPuiseaux}, we already know that $p_{2^t}(R)=1$ for any $t\in\N$. From Lemma \ref{surjection} we have $w_{2^t}(K)=1<2=s_{2^t}(k)+1$. On the other hand, if $n$ is not a power of $2$, then $x$ is not an $n$th power in $K$. Hence,
$2\leq w_n(K)\leq s_n(k)+1=2$, in other words $w_n(K)=s_n(k)+1=2$. }
\end{remark}

\begin{remark} {\rm
The assumption $s_n(Q(R))=s_n(R)$ is essential for the equality $w_n(Q(R))=s_m(k)+1$, $m>1$, to hold. Consider the ring $k[[x,\zeta_{2^t}x]]$, where $t\geq 2$, $\zeta_{2^t}$ is a primitive root of unity of order $2^t$ and $k$ is a field such that $s_2(k)>1$ (in particular, $\charr(k)\neq 2$). Then, $s_{2^{t-1}}(k)\geq s_2(k)>1$, which means that $\zeta_{2^t}\not\in k$. Consequently, $\zeta_{2^t}\not\in R$, equivalently $s_{2^{t-1}}(R)>1$. On the other hand, $Q(R)=K=k(\zeta_{2^t})((x))$. Therefore, $\zeta_{2^t}\in K$. As a result, $s_{2^{t-1}}(K)=1$. $K$ is also the field of fractions of a DVR $R'=k(\zeta_{2^t})[[x]]\supset R$ with $k(\zeta_{2^t})$ as a residue field. Applying Theorem \ref{pnQR} to $R'$, $k(\zeta_{2^t})$ and $K$ we obtain the equality $w_{2^{t-1}}(K)=s_{2^{t-1}}(k(\zeta_{2^t}))+1=2<s_{2^{t-1}}(k)+1$. }
\end{remark}

The next result gives us a sufficient condition for the equality\linebreak $s_n(Q(R))=s_n(R)$.

\begin{thm}\label{sn}
Let $K$ be a field with a nontrivial (not necessarily discrete) rank-$1$ valuation $\nu:K\to\R\cup\{\infty\}$, $R=\{f\in K:\nu(f)\geq 0\}$, $\mathfrak{m}=\{f\in K:\nu(f)>0\}$ and $k=R/\mathfrak{m}$. Take a positive integer $n$. Let $m$ be the $\charr(k)$-free part of $n$. Then $s_n(R)=s_n(K)\geq s_m(k)$, where the equalities hold if $R$ is henselian and $\charr(k)\nmid n$ or $\charr(R)=\charr(k)$.
\end{thm}

\begin{proof}
By Remark \ref{s_m}, we have $s_n(K)=s_m(K)$ in case when $\charr(R)=\charr(k)$ and $\charr(k)\mid n$. Certainly, $s_1(K)=1$. Hence, for the rest of the proof, we may assume that $n=m>1$.

Of course $s_n(K)\leq s_n(R)$. Let
\begin{equation}\label{-1}
    -1=\sum_{i=1}^l f_i^n,
\end{equation} where $f_1,\ldots f_l\in K$. If $f_1,\ldots f_l\in R$, then $l\geq s_n(R)\geq s_n(k)$, where the last inequality follows from Lemma \ref{surjection}. If not all $f_i$ belong to $R$, then without loss of generality assume that $\nu(f_l)=\min\{\nu(f_i)|i\in\{1,\ldots,l\}\}<0$. Dividing \eqref{-1} by $f_l^n$, we obtain 
\begin{equation}\label{sneq}
    -\frac{1}{f_l^n}=1+\sum_{i=1}^{l-1} \left(\frac{f_i}{f_l}\right)^n.
\end{equation}
Then,
$$-1=\frac{1}{f_l^n}+\sum_{i=1}^{l-1} \left(\frac{f_i}{f_l}\right)^n,$$
which means that $l\geq s_n(R)$. As a result, $s_n(K)\geq s_n(R)$. We can reduce equality \eqref{sneq} modulo $\mathfrak{m}$ to get $$-1=\sum_{i=1}^{l-1} \overline{\left(\frac{f_i}{f_l}\right)}^n.$$ Hence, $l\geq s_n(k)+1$. As a result, $s_n(K)\geq s_m(k)$.

If $R$ is henselian and $\charr(k)\nmid n$ or $\charr(R)=\charr(k)$, then $s_n(R)=s_n(K)=s_m(k)$ follows from Theorem \ref{pnR} b).
\end{proof}

\begin{remark} { \rm
If $R$ is not henselian, then the inequality $s_n(K)>s_m(k)$ is possible. Let $n=2$, $p$ be an odd prime number and $R=\Z_{(p)}$ be the localization of $\Z$ with respect to the prime ideal $(p)$. Then $s_2(K)=s_2(\Q)=\infty$ while $s_2(k)=s_2(\FF_p)\leq 2$.}
\end{remark}

\subsection{\textbf{\textsc{Waring numbers of rank-$1$ valuation fields.}}}
Assume for the remainder of this subsection that $K$ is a field with a nontrivial rank-$1$ valuation $\nu$ and $R=\{f\in K|\nu(f)\geq 0\}$ be henselian with maximal ideal $\mathfrak{m}=\{f\in K|\nu(f)>0\}$. For each $r\in [0,\infty)$ we define $I_r=\{f\in K|\nu(f)>r\}$. Let $n$ be a positive integer and $m$ be its $\charr(K)$-free part.

Now we give an alternative version of Theorem \ref{pnQR}, where $K$ is the field of fractions of a henselian rank-$1$ valuation ring as the following holds.

Before we state the result, we define for each local ring $A$ with maximal ideal $\mathfrak{m}$ and $g\in Q(A)$ the following value
\begin{align*}
    \ell_{n,A}^*(g)= &\inf\left\{l\in\N_+\left|\ g=\sum_{i=1}^l g_i^n \text{ for some }g_1,\ldots ,g_l\in A, g_1\not\in\mathfrak{m}\right.\right\},
\end{align*}
where we recall that $\inf\varnothing=+\infty$.

\begin{thm}\label{pnK2}
With the above assumptions, if $s_m(R)<\infty$, then for each $g\in K$ we have
\begin{align*}
    &\ \ell_{n,K}(g)=
    \inf\left\{\ell_{m,R}^*(gh_1^n), \ell_{m,R}^*\left(\frac{g}{h_2^n}\right)\left| h_1,h_2\in R, \, gh_1^n, \frac{g}{h_2^n}\in R\right.\right\}\\
    =&\ \inf\left\{\ell_{m,R/I_{2\nu(n)}}^*(\overline{gh_1^n}), \ell_{m,R/I_{2\nu(n)}}^*\left(\overline{\frac{g}{h_2^n}}\right)\left| h_1,h_2\in R, \, gh_1^n, \frac{g}{h_2^n}\in R\right.\right\}.
\end{align*}
Moreover, if $\charr(K)\nmid n$, then for every element $f \in K$ we have $\ell_n(f)<\infty$.
\end{thm}

\begin{proof}
By Proposition \ref{p_m}, we have $w_n(K)=w_m(K)$ in case of $\charr(K)\mid n$. Furthermore, it is obvious that $w_1(R)=w_1(K)=1$. Hence, for the rest of the proof, we may assume that $n=m>1$.



Let $g\in K$ be arbitrary. We will prove that $$\ell_{n,K}(g)=\inf\left\{\ell_{m,R}^*(gh_1^n), \ell_{m,R}^*\left(\frac{g}{h_2^n}\right)\left| h_1,h_2\in R, \frac{g}{h_2^n}\in R\right.\right\}.$$ Suppose that $\ell_{n,K}(g)=l$ and $g=\sum_{i=1}^l g_i^n$ for some $g_1,\ldots ,g_l\in K$. Assume without loss of generality that $\nu(g_1)=\min_{1\leq i\leq l}\nu(g_i)$. Then we write $g=g_1^n\left(1+\sum_{i=2}^l \left(\frac{g_i}{g_1}\right)^n\right)$ if $g_1\in R$ or $\frac{g}{g_1^n}=1+\sum_{i=2}^l \left(\frac{g_i}{g_1}\right)^n$ otherwise. Thus $l\geq\ell_{n,R}^*(g)$. On the other hand, if $gg_0^n=\sum_{i=1}^l g_i^n$ for some $g_0\in K\bs\{0\}$ and $g_1,\ldots ,g_l\in R$ with $g_1\not\in\mathfrak{m}$, then $g=\sum_{i=1}^l \left(\frac{g_i}{g_0}\right)^n$ and $l\geq\ell_{n,K}(g)$.

The equality $\ell_{m,R}^*(g)=\ell_{m,R/I_{2\nu(m)}}^*(\bar{g})$ follows directly from the henselianity of $R$. Hence,
\begin{align*}
    &\ \inf\left\{\ell_{m,R}^*(gh_1^n), \ell_{m,R}^*\left(\frac{g}{h_2^n}\right)\left| h_1,h_2\in R, \frac{g}{h_2^n}\in R\right.\right\}\\
    &\ =\inf\left\{\ell_{m,R/I_{2\nu(n)}}^*(\overline{gh_1^n}), \ell_{m,R/I_{2\nu(n)}}^*\left(\overline{\frac{g}{h_2^n}}\right)\left| h_1,h_2\in R, \frac{g}{h_2^n}\in R\right.\right\}.
\end{align*}
\end{proof}

\begin{thm}\label{pnK3}
With the above assumptions, we have the following inequality
$$w_n(K)\leq s_m(R/I_{2\nu(m)})+1.$$

\begin{proof}
By Proposition \ref{p_m}, we have $w_n(K)=w_m(K)$ in case of $\charr(K)\mid n$. Furthermore, it is obvious that $w_1(R)=w_1(K)=1$. Hence, for the rest of the proof, we may assume that $n=m>1$.

Since $\charr(K)\nmid n$, then $n\neq 0$ in $K$. In particular, if $a=b^n\in R\bs\{0\}$ is an $n$th power in $R$, then by henselianity of $R$ the neighbourhood $$\{c\in R|\nu(c-a)>2\nu(nb^{n-1})\}$$ of $a$ is contained in the set of $n$th powers in $R$. Let $-1=\sum_{i=1}^{s_n(R)}g_i^n$ for some $g_1,\ldots,g_{s_n(R)}\in R$. The set of $n$th powers in $R$ contains some neighbourhoods of $1$ and $g_1^n,\ldots,g_{s_n(R)}^n$. Consequently, the set of sums of $s_m(R)+1$ $n$th powers contains some neighbourhood $U$ of $0$. Now it suffices to note that every element of $K$ can be written in the form $\frac{g}{h^n}$, where $g\in U$ and $h\in R$. Thus, every element of $K$ can be written as a sum of $s_n(R)+1$ $n$th powers in $K$. By Theorem \ref{pncvR}c) we know that $s_n(R)=s_n(R/I_{2\nu(n)})$. Hence, the inequality $w_n(K)\leq s_n(R/I_{2\nu(n)})+1$ is proved.
\end{proof}
\end{thm}

\begin{cor}\label{hDVR}
Assume additionally that $R$ is a DVR with valuation group equal to $\Z$. If $s_m(R)<\infty$ and $n>2\nu(m)+1$, then $w_n(K)= s_m(R/\mathfrak{m}^{2\nu(m)+1})+1$.
\end{cor}

\begin{proof}
By Proposition \ref{p_m}, we have $w_n(K)=w_m(K)$ in case of $\charr(K)\mid n$. Hence, for the rest of the proof, we may assume that $n=m>1$.

We know from Theorem \ref{pnK3} that $w_n(K)\leq s_m(R/\mathfrak{m}^{2\nu(m)+1})+1$ (note that $\mathfrak{m}^{2\nu(m)+1}=I_{2\nu(n)}$). In order to prove our corollary, it suffices to consider any element $g\in K$ such that $\nu(g)=2\nu(n)+1$. Then $\bar{g}=\bar{0}$ and
\begin{align*}
    &\ell_{m,K}^*(g)=\ell_{m,R/\mathfrak{m}^{2\nu(n)+1}}^*(0)\\
    =&\min\left\{l\in\N_+\left|0=\sum_{i=1}^l \bar{g}_i^n\text{ for some }g_1,\ldots ,g_l\in R, g_1\not\in\mathfrak{m}\right.\right\}\\
    =&\min\left\{l\in\N_+\left|0=1+\sum_{i=2}^l \bar{g}_i^n\text{ for some }g_2,\ldots ,g_l\in R\right.\right\}\\
    =&s_n(R/\mathfrak{m}^{2\nu(m)+1})+1.
\end{align*}
This ends the proof.
\end{proof}

Theorems \ref{pnQR} and \ref{pnK3} give the following unexpected result.

\begin{cor}\label{sumof2H}
Let $R$ be a henselian local ring with the total ring of fractions $Q(R)\neq R$ and a residue field $k$. Take an odd positive integer $n>1$. Assume that $\charr(k)\nmid n$ or $R$ is rank-$1$ valuation ring with $\charr(R)\nmid n$. Then, for every element $f \in Q(R)$ there exists a presentation
$$f=f_1^n+f_2^n$$
for some $f_1,f_2 \in Q(R)$.
\end{cor}

\begin{remark} {\rm
If $\charr(R)=\charr(K)\mid n$, then it is possible that not every element of $R$ or $K$, respectively, can be written as a sum of $n$th powers in $R$ or $K$, respectively. Let $R=k[[x]]$, $K=k((x))$, $\charr(R)=\charr(K)=\charr(k)=p$ and $n=p^tm$ for some $t,m\in\N_+$ with $p\nmid m$. Denote the subfield of $p^t$th powers of $k$ by $k_t$. Then the set of sums of $l$ $n$th powers in $R$ ($K$, respectively) is the set of $l$ $m$th powers in $k_t[[x^{p^t}]]$ ($k_t((x^{p^t}))$, respectively), so it is not the whole $R$ ($K$, respectively).}
\end{remark}

\section{Waring numbers of henselian DVRs and their fields of fractions with infinite $n$th level}

\begin{thm}\label{dcvr}
Let $R$ be a DVR with the field of fractions $K$, the maximal ideal $\mathfrak{m}$ and the residue field $k$. Take a positive integer $n$ such that $s_n(k)=\infty$. Then $$w_n(K)=w_n(R)\geq w_n(k),$$ where the equality holds if $R$ is henselian.
\end{thm}

\begin{proof}
    
    
    
    
Since $s_n(k)=\infty$, we deal with a real field $k$. In particular, $\charr(R)=\charr(k)=0$. 

The inequality $w_n(R)\geq w_n(k)$ follows from Lemma \ref{surjection}. 

Let $f\in K$ be a sum of $n$th powers of elements in $K$. $f$ can be written as $f=\pi^{tn}f_0$, where $\pi$ is a uniformizer of $R$, $f_0 \in R$ with $\nu(f_0)=0$ and $t\in\Z$. This is because $0$ cannot be written as a sum of finitely many nonzero $n$th powers in $k$. We thus have $\ell_n(f)=\ell_n(f_0)$, where we consider the $n$-length with respect to $K$. If $f_0=\sum_{i=1}^l g_i^n$ for some $g_1,\ldots,g_l\in K$, then $g_1,\ldots,g_l\in R$ as $0$ cannot be written as a sum of finitely many nonzero $n$th powers in $k$. This shows that $w_n(K)=w_n(R)$. If we assume additionally that $R$ is henselian, then it follows that $\ell_n(f_0)=\ell_n(\Bar{f}_0)$. This proves the equality $w_n(R)=w_n(k)$.
\end{proof}

\begin{remark} {\rm
The assumption in the above theorem that $R$ is a DVR is essential. We know that the ring $R_s=\R[[x_1,\ldots,x_s]]$, $s\geq 1$, is henselian and
$$w_2(R_s)=\begin{cases}
1,&\text{ if }s=1\\
2,&\text{ if }s=2\\
\infty,&\text{ if }s>2
\end{cases}$$
(cf. \cite{CLDR}) meanwhile $s_2(\R)=\infty$ and $w_2(\R)=1$. Moreover, we have $K_s=\R((x_1,\ldots,x_s))$, $s\geq 1$, and 
$w_2(K_1)=1$, $w_2(K_2)=2$ (cf. \cite[Corollary 5.4]{CLDR}), $w_2(K_3)=4$ (cf. \cite[Theorem 1.2]{hu}). These are the only known values of $w_2(K_s)$. As was mentioned in the Introduction, it is known
 that $w_2(\R(x_1,\dots,x_s) \leq 2^s$. However, under the assumption $w_2(\R(x_1,\dots,x_s))=2^s$, one can show $w_2(\R((x_1,\dots,x_{s+1})))=2^s$ (see \cite[Corollary 2.3]{hu}).}


\end{remark}

\section{Applications}
In this Section we provide some applications of the results proved before. Firstly, we consider the Waring numbers of the rings of formal power series. Secondly, we obtain some lower bound on Waring numbers of thee coordinate ring of an irreducible algebraic set. Next, we deal with the problem of computing Waring numbers of the ring of $p$-adic integers $\Z_p$ and its field of fractions $\Q_p$. At the end, we study the relations between Waring numbers of local rings and their henselizations and completions. 






\subsection{Waring numbers of the rings of power series.}

\begin{thm}\label{series}
Let $k$ be a field, $n,s$ be positive integers and $m$ be a $\charr(k)$-free part of $n$. 
\begin{enumerate}
    \item[a)] We have
    \begin{equation}\label{eq0}
        s_n(k[[x_1,\dots x_s]])=s_n(k((x_1,\dots x_s)))= s_n(k((x_1))\dots ((x_s)))=s_m(k).
    \end{equation}
    \item[b)] If $s_n(k)<\infty$, then
    \begin{equation}\label{eq1}
        w_n(k[[x_1,\dots, x_s]])= \begin{cases}
        \max \{w_m(k), s_m(k)+1 \} & \text{ for } m>1    \\
        1 & \text{ for } m=1
        \end{cases},
    \end{equation}
    \begin{equation}\label{eq2}
    w_n(k((x_1,\dots x_s)))=w_n(k((x_1))\dots ((x_s)))=\begin{cases}
    s_m(k)+1 & \text{ for } m>1\\
    1 & \text{ for } m=1 \end{cases},
    \end{equation}
    \begin{equation}\label{eq3}
         w_n(k((x_1))\dots ((x_s))[[x_{s+1}]])=\begin{cases}
        s_m(k)+1 & \text{ for } m>1    \\
        1 & \text{ for } m=1
        \end{cases}.
    \end{equation}
    
\end{enumerate}
\end{thm}

\begin{proof}
By Proposition \ref{p_m} we have $s_m(k)=s_n(k)$.

Equality \eqref{eq1} follows from Theorem \ref{pnR}. 

By simple induction on $s$ and use of Theorem \ref{sn} one can show that
$$s_n(k((x_1))\dots ((x_s)))=s_n(k).$$ 
Moreover, we have the following inclusions
$$k\subset k[[x_1,\dots,x_s]]\subset k((x_1,\dots,x_s)) \subset k((x_1))\dots ((x_s))$$
which implies the inequalities
$$s_n(k)\geq s_n(k[[x_1,\dots,x_s]]) \geq s_n(k((x_1,\dots,x_s)))\geq s_n(k((x_1))\dots((x_s))). $$
This proves \eqref{eq0}.

The formula for $w_n(k((x_1,\dots,x_s)))$ (resp. $w_n(k((x_1))\dots ((x_s)))$ ) follows from Theorem \ref{pnQR} by taking $R=k[[x_1,\dots,x_s]]$ (resp. $R=k((x_1))\dots ((x_{s-1}))[[x_s]])$ and $R_\nu$ where $\nu=\nu_{x_s}$ is the $x_s$-adic valuation. Clearly $k[[x_1,\dots,x_s]] \subset R_\nu$ (resp. $k((x_1))\dots ((x_{s-1}))[[x_s]] \subset R_\nu$), hence the aforementioned Theorem can be applied.

Equality \eqref{eq3} follows from Theorem \ref{pnR} and formulae \eqref{eq0} and \eqref{eq2}.
\end{proof}

\begin{thm}\label{seriesinfinitelevel}
Let $k$ be a field such that $s_n(k)=\infty$. Then the following holds
\begin{equation}\label{eq5}
w_n(k((x_1))\dots ((x_{s-1}))[[x_s]])=w_n(k((x_1))\dots ((x_s)))=w_n(k).
\end{equation}
\begin{equation}\label{eq6}
        w_n(k[[x_1,\ldots,x_s]])\geq w_n(k((x_1,\ldots,x_s)))\geq w_n(k),
    \end{equation}
\end{thm}
\begin{proof}
The equality \eqref{eq5} follows directly from Theorem \ref{dcvr}.
For the proof of \eqref{eq6} let us see that the inequality $$w_n(k[[x_1,\ldots,x_s]])\geq w_n(k((x_1,\ldots,x_s)))$$ is clear. Consider now any element $a \in k$ of finite $n$-length. Since $k$ is the residue field of $k[[x_1]]$, then by the proof of Theorem \ref{dcvr}, the number $\ell_n(a)$ remains constant in $k$, $k[[x_1]]$ and $k((x_1))$, hence by a simple induction, also in $k((x_1))\dots ((x_s))$.
We have the following chain of inclusions
$k \subset k[[x_1,\dots,x_s]] \subset k((x_1,\dots,x_s))\subset k((x_1))\dots((x_s)).$

By Lemma \ref{surjection}, $n$-length of an element cannot increase via a homomorphism of rings. By the above discussion, $n$-length of any element $a \in k$ is the same in $k$ and in $k((x_1))\dots ((x_s))$, hence also in $k((x_1,\dots,x_s))$. The proof of \eqref{eq6} is finished.

\end{proof}




\subsection{Waring numbers of coordinate rings.}
Let $V\subset k^s$ be an algebraic set. Denote by $I(V)\subset k[x_1,x_2,\dots, x_s]$ the ideal of polynomials vanishing on $V$. Define 
$$k[V]:=k[x_1,x_2,\dots, x_s]/I(V),$$
to be the coordinate ring of $V$. We put $k(V)$ to be the field of fractions of $k[V]$, provided that $V$ is an irreducible algebraic set.

We say that the point $x\in V$ is a regular point, if the ring $k[V]_{\mathfrak{m}_x}$ is a regular local ring, where $\mathfrak{m}_x$ is the maximal ideal of polynomial functions vanishing in $x$.
For rudiments of commutative algebra we refer the reader to \cite[Chapter 7,19]{eisenbud}.

\begin{thm}\label{Main Theorem}
Let $V$ be an irreducible algebraic subset of $k^s$, different from the point, which admits a regular point. 
\begin{itemize}
    \item[a)] If $s_n(k)<\infty$, then
$$w_n(k[V])\geq \max \{w_{m}(k), s_{m}(k)+1 \},$$
$$w_n(k(V)) \geq s_m(k)+1,$$
where $m$ is the $\charr(k)$-free part of $n$.
    \item[b)] If $s_n(k)=\infty$, then
    $$w_n(k[V]) \geq w_n(k(V))\geq  w_n(k).$$
\end{itemize}

\begin{proof}
The ring $k[V]$ is an integral domain, as $V$ is irreducible. Hence, if $\charr(k)>0$ then the Frobenius endomorphism is injective, so by Proposition \ref{p_m}, we may assume $n=m$.

Let $p$ be a regular point of $V$ and $\mathfrak{m}=(f_1,f_2,\dots, f_l)$ be a maximal ideal of regular functions vanishing at $p$, with the smallest possible $l$. In order to prove the thesis it is enough to find elements of appropriate $n$-length.
Consider the following commutative diagram.

\begin{tikzcd}
k[V] \arrow{d} \arrow{rr} \arrow[rrdd, "\psi"] & &  k(V) \arrow{dd} \\
k[V]_\mathfrak{m} \arrow[d, "j"]   & & \\
 \widehat{k[V]_{\mathfrak{m}}} \arrow[r, "\phi"] & k[[f_1,f_2,\dots,f_l]]\arrow[r,"\iota"] & k((f_1))((f_2))\dots ((f_l)) \\
\end{tikzcd}

Every map in the above diagram is either monomorphism or isomorphism, which we will now justify.
Take the regular local ring $k[V]_{\mathfrak{m}}$ and  let $\widehat{k[V]_{\mathfrak{m}}} $ be the $\mathfrak{m}$-adic completion of $k[V]_{\mathfrak{m}}$. The map $j$ is a monomorphism, since we are completing a local ring along its maximal ideal.
Completion of a regular local ring is again a regular local ring. By the Cohen Structure Theorem \cite[Theorem 15]{cohen}, every complete regular local ring containing a field is isomorphic to the ring of formal power series over this field. In our case 
$\widehat{k[V]_{\mathfrak{m}}} \cong k[[f_1,f_2,\dots, f_l]] $. Hence, $\phi$ is an isomorphism, which can be chosen to preserve generators of the maximal ideal.
Consequently, we can embed the ring $k[[f_1,f_2,\dots, f_l]]$ into the field of iterated Laurent series $k((f_1))((f_2))\dots ((f_l))$. 

By Lemma \ref{surjection}, the $n$-length of an element cannot increase via a monomorphism.
If $w_n(k)=\infty$ take any positive integer $D$ and $a \in k$ such that $\ell_{n,k}(a)>D$. If $w_n(k)< \infty$, it is enough to take $a \in k$ such that $ \ell_{n,k}(a)=w_n(k) $. Part a) follows from considering the images of $a$ and $f_l$ in the ring $k[[f_1,f_2,\dots, f_l]]$ and $f_l$ in $k((f_1))((f_2))\dots ((f_l))$. Part b) follows from considering the image of $a$ in $k((f_1))((f_2))\dots ((f_l))$. 
\end{proof}
\end{thm}

In case of $s_n(k)<\infty$, if we further assume that $\charr (k)$ is coprime with $n!$, then $w_n(k(V))$ is finite (cf. Introduction).

\begin{remark} {\rm
Theorem \ref{Main Theorem} is not sharp if $s_n(k)=\infty$. If $\dim V \geq 3$, then $w_2(k[V])=\infty$ (see \cite{CLDR}) meanwhile $w_2(k)$ may be finite. }
\end{remark}

\begin{cor}\label{algset2}
Let $V$ be an irreducible algebraic subset of $k^s$, different from the point, which admits a regular point. Assume that $\charr (k) \neq 2$ and $s_2(k)<\infty$. Then $w_2(k[V])=w_2(k(V))=s_2(k)+1$.
\begin{proof}

Let $a$ be an element in $k[V]$ or $k(V)$. We then have
$$a=\left(\frac{a+1}{2}\right)^2-\left(\frac{a-1}{2}\right)^2.$$
Hence, the $2$-length of any element is at most $s_2(k)+1$. The lower bound follows directly from Theorem \ref{Main Theorem}.

\end{proof}
\end{cor}

\begin{remark} {\rm
If $V$ is an irreducible algebraic subset of $k^s$, different from the point, which admits a regular point, and $\charr (k) = 2$, then by Proposition \ref{free part} we have $$w_2(k(V))=w_2(k[V])=w_1(k(V))=w_1(k[V])=1.$$ }
\end{remark}

\subsection{Waring numbers of rings and fields of $p$-adic numbers}\label{p-adic}
In this Section we discuss the Waring numbers of $p$-adic rings and fields. We stress that the results concerning $p$-adic integers are not new. Indeed, they can be easily derived from papers \cite{hardy littlewood 1, hardy littlewood 2, hardy littlewood 3} of Hardy and Littlewood on analytic number theory. It is important to note that the $\Gamma(k)$ from the above references can be rephrased in terms of the Waring numbers in the following way (cf. \cite{elsholtz})
$$\Gamma(k)=\max_{p-prime} w_k(\Z_p).$$
On the other hand, the authors are not aware if the results concerning $p$-adic fields are present in the literature (or can be easily deduced).

To keep the paper self contained as much as possible, we include algebraic proofs of all the facts presented here.

At first, we present a result on $n$th Waring numbers of fields of fractions of henselian DVRs, where $n$ is odd.

\begin{thm}\label{dvr -odd}
Let $R$ be a henselian DVR with quotient field $K$ and $n>1$ be an odd positive integer. Denote by $m$ the $\charr(K)$-free part of $n$. Then

$$w_n(K)=
\begin{cases}
1 \ \ if \ \   m=1, \\
2 \ \ if  \ \  m>1. \\
\end{cases}
$$

\begin{proof}

By Proposition \ref{p_m}, we have $w_n(K)=w_m(K)$ in case of $\charr(K)\mid n$. Hence, for the rest of the proof, we may assume that $n=m$.

If $n=1$, then clearly $w_n(K)=1$. If $n>1$, then by Theorem \ref{pnK3} we have $w_n(K)\leq 2.$ Since the uniformizer cannot be an $n$th power, we obtain $w_n(K)=2.$

\end{proof}
\end{thm}

\begin{remark} {\rm
In particular, for any prime number $p$ and any odd integer $n$ we have $w_n(\Q_p)=2$. }
\end{remark}

Now, we give formulas for $n$th Waring numbers of $\Z_p$ and $\Q_p$ in some cases, when $p\mid n$.

\begin{thm}\label{phipk}
Let $p$ be an odd prime number, $k$ be a positive integer and $d$ be a positive integer not divisible by $p$. Then 
\begin{equation*}
w_{dp^{k-1}(p-1)}(\Z_p)=w_{dp^{k-1}(p-1)}(\Q_p)=p^k.
\end{equation*}
\end{thm}

\begin{proof}
At first, let us note that the set of powers of degree $dp^{k-1}(p-1)$ in the group $U(\Z_p)$ of units of $\Z_p$ is exactly the set of $p$-adic integers congruent to $1$ modulo $p^k$. Indeed, by Euler's theorem we have $c^{dp^{k-1}(p-1)}\equiv 1\pmod{p^k}$ for each integer $c$ not divisible by $p$. Furthermore, the subgroup of powers of degree $dp^{k-1}(p-1)$ in the group $U(\Z/(p^l))\cong\Z/(p^{l-1}(p-1))$ of units of the ring $\Z/(p^l)$, where $l\geq k$, has $\frac{p^{l-1}(p-1)}{\gcd (dp^{k-1}(p-1),p^{l-1}(p-1))}=p^{l-k}$ elements. This means that every congruence class modulo $p^l$, $l\geq k$, coprime to $p$ and being a power of degree $dp^{k-1}(p-1)$ splits into $p$ congruence classes modulo $p^{l+1}$ being powers of degree $dp^{k-1}(p-1)$. This, combined with the fact that $\Z_p$ is the inverse limit of rings $\Z/(p^l)$, gives that $a=b^{dp^{k-1}(p-1)}$ for some $b\in U(\Z_p)$ if and only if $a\equiv 1\pmod{p^k}$.

Let $l=2k-1$. Then $dp^{k-1}(p-1)>2k-1$ and consequently $c^{dp^{k-1}(p-1)}\equiv 0\pmod{p^{2k-1}}$ for each $c$ divisible by $p$. Thus $$w_{dp^{k-1}(p-1)}(\Z/(p^{2k-1}))=s_{dp^{k-1}(p-1)}(\Z/(p^{2k-1}))=p^k-1.$$ Applying Corollaries \ref{pnHDVR} and \ref{hDVR} we get the result.
\end{proof}

\begin{thm}\label{pnQp}
Let $p$ be an odd prime number, $k$ be a positive integer and $d$ be an odd positive integer not divisible by $p$. Assume additionally that $\frac{dp^{k-1}(p-1)}{2}>1$. Then the following equalities hold:
\begin{align*}
  w_{\frac{dp^{k-1}(p-1)}{2}}(\Z_p) &\ =\frac{p^k-1}{2},  \\
  w_{\frac{dp^{k-1}(p-1)}{2}}(\Q_p) &\ =2.
\end{align*}
\end{thm}

\begin{proof}
Analogously as in the proof of Theorem \ref{phipk} we show that the set of powers of degree $\frac{dp^{k-1}(p-1)}{2}$ in $U(\Z/(p^l))$, $l\geq k$, is the set of classes modulo $p^l$ congruent to $\pm 1$ modulo $p^k$. Moreover, $\frac{dp^{k-1}(p-1)}{2}\geq 2k-1$, which means that $c^{\frac{dp^{k-1}(p-1)}{2}}\equiv 0\pmod{p^{2k-1}}$ for each $c$ divisible by $p$. Thus $w_{\frac{dp^{k-1}(p-1)}{2}}(\Z/(p^{2k-1}))=\frac{p^k-1}{2}$ and $s_{\frac{dp^{k-1}(p-1)}{2}}(\Z/(p^{2k-1}))=1$ and the results follow from Theorem \ref{pnK2} and Corollary \ref{hDVR}.
\end{proof}

\begin{thm}\label{pnQ2}
Let $k,d$ be positive integers, with $d$ odd. Then the following holds
$$
w_{2^kd}(\Z_2)=w_{2^kd}(\Q_2)=
\begin{cases}
4 & \text{if } k=1, d=1 \\
15 & \text{if } k=2, d=1 \\
2^{k+2} & \text{if } k>2\text{ or }d\geq 3 \\
\end{cases}.$$
\end{thm}

\begin{proof}
Just as in the proof of Theorem \ref{phipk} we show that the set of powers of degree $2^kd$ in $U(\Z/(2^l))$, $l\geq k+2$, is the set of classes modulo $p^l$ congruent to $1$ modulo $2^{k+2}$ (here we note that $U(\Z/(2^l))\cong\Z/(2^{l-2})\times\Z/(2)$ for $\l\geq 2$). If $k>2$ or $d\geq 3$, then $2^kd>2k+1$, which means that $c^{2^kd}\equiv 0\pmod{2^{2k+1}}$ for each even $c$. Thus $w_{2^kd}(\Z/(2^{2k+1}))=s_{2^kd}(\Z/(2^{2k+1}))=2^{k+2}-1$. By Corollaries \ref{pnHDVR} and \ref{hDVR} we have
\begin{align*}
    & w_{2^kd}(\Z_2)=w_{2^kd}(\Q_2)=\max\{w_{2^kd}(\Z/(2^{2k+1})),s_{2^kd}(\Z/(2^{2k+1}))+1\}\\
    = & 2^{k+2}.
\end{align*}
If $k\in\{1,2\}$ and $d=1$, then $2^k<2k+1$. By Corollary \ref{pnHDVR} we then have $$w_{2^k}(\Z_2)=w_{2^k}(\Z/(2^{2k+1}))=
\begin{cases}
4 & \text{if } k=1,\\
15 & \text{if } k=2.
\end{cases}$$
as the powers of degree $2^k$ in $\Z/(2^{2k+1})$ are $0$, $1$ and $2^{2k}$. In order to obtain the equality 
$$w_{2^k}(\Q_2)=
\begin{cases}
4 & \text{if } k=1,\\
15 & \text{if } k=2
\end{cases}$$
we use Theorem \ref{pnK2}. We check the values of $\ell_{2^k,\Z/2^{2k+1}}^*(x)$, where $x\in\Z/2^{2k+1}$. We easily conclude that $$w_{2^k}(\Q_2)=\ell_{2^k,\Z/2^{2k+1}}(2^{2k+1}-1)=\ell_{2^k,\Z/2^{2k+1}}^*(2^{2k+1}-1).$$
\end{proof}

As a direct corollary from Theorem \ref{pnHDVR}, Theorem \ref{sn}, Corollary \ref{hDVR} and Theorem \ref{pnQ2} we can give the following result that comes down the problem of computing $w_n(\Z_p)$, $s_n(\Z_p)$, $w_n(\Q_p)$ and $s_n(\Q_p)$ to computing $w_n(\Z/p^{2\nu_p(n)+1})$ and $s_n(\Z/p^{2\nu_p(n)+1})$.

\begin{cor}
Let $p$ be a prime number and $n>1$ be a positive integer. Then $s_n(\Z_p)=s_n(\Q_p)=s_n(\Z/p^{2\nu_p(n)+1})$,
$$w_n(\Z_p)=\begin{cases}
4 & \text{if } (p,n)=(2,2),\\
15 & \text{if } (p,n)=(2,4),\\
\max\{w_n(\Z/p^{2\nu_p(n)+1}),s_n(\Z/p^{2\nu_p(n)+1})+1\} & \text{otherwise}
\end{cases}$$ 
and
$$w_n(\Q_p)=\begin{cases}
4 & \text{if } (p,n)=(2,2),\\
15 & \text{if } (p,n)=(2,4),\\
s_n(\Z/p^{2\nu_p(n)+1})+1 & \text{otherwise}.
\end{cases}$$ 
\end{cor}

The following theorem shows that computing $n$th Waring number for $\Z_p$ and $\Q_p$ is very easy if $p$ is large comparing to $n$.

\begin{thm}
Let $n$ be a positive integer. Then for any prime $p$ satisfying $p>(n-1)^4$ we have the following formulas
$$w_n(\Z_p)=w_n(\Q_p)=
\begin{cases}
2 & \text{if } \gcd(n,p-1) \, | \, \frac{p-1}{2}, \\
3 & \text{otherwise}. \\
\end{cases}$$
\end{thm}
\begin{proof}
In view of the Corollary \ref{pnR=...}  and Theorem \ref{pnQR} it is enough to compute the values of $w_n(\FF_p)$ and $s_n(\FF_p)$ for $p>(n-1)^4$. By \cite{small2} we get
$$w_n(\FF_p)=
\begin{cases}
1 & \text{if } \gcd(n,p-1)=1 \\
2 & \text{if } \gcd(n,p-1)>1 \\
\end{cases}
$$
and 
$$s_n(\FF_p)=
\begin{cases}
1 & \text{if } \gcd(n,p-1) \, | \, \frac{p-1}{2}  \\
2 & \text{otherwise.} \\
\end{cases}
$$
By combining the above formulas we get the thesis. 

\end{proof}
 \begin{remark}
{\rm The above Theorem implies that in order to compute $w_n(\Z_p)$ and $w_n(\Q_p)$ for all primes $p$ and a fixed positive integer $n$, one has to consider only finitely many exceptional cases.}
 \end{remark}
 
 It is worth to mention the results of Voloch \cite{voloch} on Waring numbers of $p$-adic rings. He was considering unramified extensions of $\Z_p$, yet we present here the result for $\Z_p$.
 \begin{thm}\cite[Corollary]{voloch}
 Let $n=pd$ where $(p,d)=1$. If $p \geq \max\{27d^6, 13 \}$ then
 $w_n(\Z_p)\leq 9$. If we further assume that $n$ is odd then 
 $w_n(\Z_p) \leq 8$.
 
 \end{thm}

Below we present full classification of $w_n(\Z_p)$ and $w_n(\Q_p)$ for $n=3,4,5$ and any prime number $p$. By Theorem \ref{dvr -odd} $w_3(\Q_p)=w_5(\Q_p)=2$, hence this values are omitted from the tables. The data below follows from Theorem  \ref{pncvR} and Corollary \ref{pnR=...}. The values of $w_n(\FF_p)$ and $s_n(\FF_p)$ are taken from \cite{small}.
\begin{center}
\begin{table}[H]
\begin{tabular}{  |c|c|c|c| } 
 \hline
 $p$ & $w_3(\Z_p)$ & $w_3(\FF_p)$  & $s_3(\FF_p)$\\ 
 \hline
 3  & 4  & 1 & 1 \\ 
 7 & 3 & 3 & 1 \\ 
 $p \equiv 1\pmod 3$, $p\neq 7$ &  2 & 2 & 1 \\
 $p \equiv 2\pmod 3$ &  2 & 1 & 1 \\
 \hline
\end{tabular}
\caption{\label{tab:table-name}$w_3(\Z_p)$, of course $w_3(\Q_p)=2$ for any prime $p$.}
\end{table}
\end{center}

\begin{center}
\begin{table}[H]
\begin{tabular}{  |c|c|c|c|c| } 
 \hline
 $p$ & $w_4(\Z_p)$ & $w_4(\Q_p)$ & $w_4(\FF_p)$  & $s_4(\FF_p)$\\ 
 \hline
2 & 15 & 15 & 1 & 1 \\ 
5 & 5 & 5 & 4 & 4\\ 
  13 & 3 & 3 & 3 &2 \\ 
 29 & 4 & 4 & 3 & 3\\ 
 17,41 & 3 & 2 & 3 & 1 \\ 
 37,53,61 & 3 & 3 & 2 &2 \\ 
73 & 2 & 2 & 2 & 1  \\ 
$p\equiv 3 \pmod 4$, $p<81$ & 3 & 3 & 2 & 2 \\ 
 $p\equiv 1 \pmod 8$, $p>81$ & 2 & 2 & 2 & 1 \\ 
 $p \not \equiv 1 \pmod 8$, $p>81$ & 3 & 3 & 2 & 2\\ 

 \hline
\end{tabular}
\caption{\label{tab:table-name}$w_4(\Z_p)$ and $w_4(\Q_p)$.}
\end{table}
\end{center}

\begin{center}
\begin{table}[H]
\begin{tabular}{  |c|c|c|c| } 
 \hline
 $p$ & $w_5(\Z_p)$ & $w_5(\FF_p)$  & $s_5(\FF_p)$\\ 
 \hline
 5 & 3 & 1 & 1 \\ 
 11& 5 & 5 & 1 \\ 
   $ p\not\equiv 1 \pmod 5$& 2 & 1 & 1\\
  $p\equiv 1 \pmod 5$, $p\geq 131$ & 2& 2& 1\\
  $p\equiv 1 \pmod 5$, $p<131$, $p\neq11$ & 3 & 3 & 1\\
 \hline
\end{tabular}
\caption{\label{tab:table-name}$w_5(\Z_p)$, of course $w_5(\Q_p)=2$ for any prime $p$.}
\end{table}
\end{center}

\begin{remark}
{\rm Let us denote for any ring $R$ and positive integers $l,n$ the set
$$S_l^n(R)=\left\{\left.\sum_{i=1}^l x_i^n\right| x_1,\ldots x_l\in R\right\}.$$
Since $\Z$ and $\Q$ are dense (with respect to $p$-adic topology) in $\Z_p$ and $\Q_p$, respectively, we know that $S_l^n(\Z)$ and $S_l^n(\Q)$ are dense in $S_l^n(\Z_p)$ and $S_l^n(\Q_p)$, respectively. On the other hand, from henselianity of $\Z_p$ it follows that $S_l^n(\Z_p)$ and $S_l^n(\Q_p)$ are closed subsets of $\Z_p$ and $\Q_p$, respectively. Hence, we get the result that $S_l^n(\Z)$ is dense in $\Z_p$ if and only if $l\geq w_n(\Z_p)$ and $S_l^n(\Q)$ is dense in $\Q_p$ if and only if $l\geq w_n(\Q_p)$. In particular, the minimal value of $l$ such that $S_l^n(\Q)$ is dense in $\Q_p$ agrees with the minimal value of $l$ such that $R(S_l^n(\Z))$ is dense in $\Q_p$ for all pairs $(p,n)$ except for $(2,2),(2,4)$ and $(2,8)$ , where
$$R(A)=\left\{\left.\frac{a}{b}\right| a,b\in A, b\neq 0\right\}$$
(see \cite[Theorem 1.9]{miska}).}
\end{remark}


\subsection{Waring numbers of local rings and their henselizations and completions}
In this subsection we provide lower bounds on the Waring numbers of a local ring $(R, \mathfrak{m})$, under some mild assumptions on $R$ and its $\mathfrak{m}$-adic completion. Assume that $n$ is a positive integer greater than 1. Denote by $(R^h, \mathfrak{m}^h)$ the henselization of the local ring $R$ and let $(\widehat{R}, \widehat{m})$ be the $\mathfrak{m}$-adic completion of $R$.

\begin{thm}
Let $(R, \mathfrak{m})$ be a local ring with residue field $k$, maximal ideal $\mathfrak{m}\neq \mathfrak{m}^2$ and $s_n(k)<\infty$. Assume that $\charr (k) \nmid n $ or $\charr (R) = \charr (k)=p$ and the rings $R^h, \widehat{R}$ are reduced. Then
\begin{enumerate}
    \item[a)] $s_n(R)\geq s_n(R^h)=s_n(\widehat{R})$
    \item[b)] $w_n(R)\geq w_n(R^h)=w_n(\widehat{R}).$
\end{enumerate}

\begin{proof}

The ring $\widehat{R}$ is henselian and has the same residue field as $R^h$, hence by Theorem \ref{pnR} and Corollary \ref{pnR=...} $s_n(R^h)=s_n(\widehat{R})$ and $w_n(R^h)=w_n(\widehat{R})$. The inequality $s_n(R)\geq s_n(\widehat{R})$ follows directly from Lemma \ref{surjection}. The last inequality follows from Theorem \ref{pnR} c). 
\end{proof}
\end{thm}

\begin{remark} {\rm
The assumption that $\widehat{R}$ is reduced is necessary for the above proof to work, as there exists an integral local ring of dimension 1 whose completion is not reduced, cf. \cite[Appendix 1, Example 5]{nagata}.}
\end{remark}

If we make an assumption that the $R$ is a DVR, we get the following.

\begin{thm}
Let $(R,\mathfrak{m})$ be a DVR, $(R^h,\mathfrak{m}^h)$ and $(\widehat{R},\widehat{\mathfrak{m}})$ be its henselization and completion, respectively.

Then the following inequality holds.
$$w_n(R)\geq w_n(R^h)=w_n(\widehat{R})$$

If we denote by $K$, $K^h$ and $\widehat{K}$ their fields of fractions, respectively, then

$$w_n(K)\geq w_n(K^h)=w_n(\widehat{K}).$$

\begin{proof}
The equality $w_n(R^h)=w_n(\widehat{R})$ follows from henselianity and Theorem \ref{pncvR}. The remaining inequality 
$w_n(R)\geq w_n(\widehat{R})$ follows from the Theorem \ref{pncvR} and elementary properties of completion. The second part follows from Theorem \ref{pnQR}.
\end{proof}

\end{thm}

We may now state the following definition.

\begin{df}
Let $R$ be a ring and $n>1$ be a positive integer. We say that a prime ideal $\mathfrak{p} \subset R$ is an $n$-good ideal if $\mathfrak{p} R_\mathfrak{p}\neq (\mathfrak{p} R_\mathfrak{p})^2$, $s_n(R_\mathfrak{p}/\mathfrak{p}R_\mathfrak{p})<\infty,$  and one of the following conditions hold:
\begin{enumerate}
\item $\charr (R_\mathfrak{p}/\mathfrak{p}R_\mathfrak{p}) \nmid n$
\item $\charr(R_\mathfrak{p})=\charr (R_\mathfrak{p}/\mathfrak{p}R_\mathfrak{p})=p \mid n$ and the $\mathfrak{p}R_\mathfrak{p}$-adic completion of $R_\mathfrak{p}$ is reduced.
\item $R_\mathfrak{p}$ is a DVR.
\end{enumerate}
\end{df}

The above discussion readily implies the following.

\begin{thm}
Let $R$ be a ring and $n>1$ be a positive integer. Then the following inequality holds
$$w_n(R) \geq  \sup w_n(\widehat{R_\mathfrak{p}})$$
where supremum runs over all $n$-good ideals of $R$.
\end{thm}

If $R$ is an integral domain, the zero ideal is not considered an  $n$-good ideal as $(0)=(0)^2$, but we may obtain rather trivial inequality $w_n(R)\geq w_n(R_{(0)})=w_n(Q(R))$.
\begin{remark}\label{remark problems} {\rm
It is not clear whether the notion of an $n$-good ideal is necessary. Let $(R,\mathfrak{m})$ be a henselian local ring and $n>1$ be a positive integer.
\begin{enumerate}
\item Assume that $\mathfrak{m}\neq \mathfrak{m}^2$ and $\charr (R)=\charr(k) \mid n$. In this case, nilpotents may occur in the completion, and we do not know how to compute $w_n(R)$. This is because the Frobenius map is not injective. 
\item Assume that $\mathfrak{m}\neq \mathfrak{m}^2$ and $\charr (R)\neq \charr(k) \mid n$. This is equivalent to $n \in \mathfrak{m}$. Here, we are not able to compute $w_n(R)$ if $R$ is NOT a DVR. In particular, our theory cannot be applied to the case $n=p$ and $R=\Z_p[[x]]$, yet Theorem \ref{joly} implies finiteness of $w_p(\Z_p[[x]])$. However if $n$ is a multiple of $p$, neither Theorem \ref{joly} nor theory developed in this paper can be applied.

\item The last case deals with arbitrary $n$ and $\mathfrak{m}=\mathfrak{m}^2$. In most cases, Theorem \ref{pnR} gives us an upper bound for $w_n(R)$, however by Remark \ref{almostPuiseaux}, these bounds may be sharp. Completion of such a ring degenerates into the residue field. Hence, it is possible $w_n(R^h) > w_n(\widehat{R})$.

\end{enumerate}}
\end{remark}
As a consequence we propose two problems

\begin{prob}
Let $(R,\mathfrak{m})$ be a local ring and $n>1$ be a positive integer. Denote by $(R^h, \mathfrak{m}^h)$ the henselization of the local ring $R$ and let $(\widehat{R}, \widehat{m})$ be the $\mathfrak{m}$-adic completion of $R$. Assume that $\mathfrak{m}\neq \mathfrak{m}^2$. Is it always true that
$w_n(R)\geq w_n(R^h)=w_n(\widehat{R})$?
\end{prob}

\begin{prob}
Let $(R,\mathfrak{m})$ be any local ring. Does the following double inequality 
$w_n(R)\geq w_n(R^h)\geq w_n(\widehat{R})$
hold?
\end{prob}

With this in mind, we can now apply the above theory to the very specific class of rings, namely the almost Dedekind domains. Recall, that the integral domain $R$ is called an almost Dedekind domain, if for every maximal ideal $\mathfrak{m} \subset R$, the localization $R_\mathfrak{m}$ is a DVR. In particular, a Noetherian almost Dedekind domain is a Dedekind domain.

\begin{cor}
Let $R$ be an almost Dedekind domain with a fraction field $K$ and $n>1$ be a positive integer. Then, the following inequalities hold:
$$w_n(R) \geq \sup_{\mathfrak{m}} w_n(\widehat{R_\mathfrak{m}}) $$
and 
$$w_n(K) \geq \sup_{\mathfrak{m}} w_n(\mathrm{Frac}(\widehat{R_\mathrm{m}})), $$
where supremum is taken over all maximal ideals.
\end{cor}
 
 In the case of a ring of integers $\OO_K$ of a number field $K$ (which is a Dedekind domain, of course), we can obtain lower bounds by considering $p$-adic rings and fields.
 
 \begin{cor}\label{numberfield}
 Let $K$ be a number field with its ring of integers $\OO_K=\Z[\alpha_1,...,\alpha_s]$. Then the following inequalities hold
 $$w_n(\OO_K) \geq \sup_{p - \mathrm{prime}} w_n(\Z_p[\alpha_1,...,\alpha_s])$$
 $$w_n(K)\geq \sup_{p - \mathrm{prime}} w_n(\Q_p[\alpha_1,...,\alpha_s]).$$
 \end{cor}

 If $n$ is even, the problem of finding $w_n(\Z)$ is known as the Waring problem and is (almost) completely solved. 
 If $n$ is odd, we get lower bounds for the easier Waring problem (cf. \cite{chinburg, wright}). Currently, the best known lower bounds follow from considering residues modulo a prime number to an appropriate power. This is precisely the same way as in Corollary \ref{pnHDVR}, hence we do not obtain any new results for $\Z$.
 
 For the rational numbers $\Q$, the above result gives a new insight into lower bounds for even $n$ (for odd $n$ we get trivial bound, as the right hand side is equal to 2).
 \begin{thm} For $n=4,6,8$ we have the following lower bounds:
 \begin{enumerate}
 \item $w_4(\Q) \geq 15$\footnote{ The authors found that $w_4(\Q)\in\{15,16\}$, which can be deduced from the consideration directly before Theorem 395, pp. 427-428 and the note in p. 446 in \cite{hardy}},
\item $w_6(\Q)\geq 9$,
 \item $w_8(\Q) \geq 32$.
\end{enumerate}
\begin{proof}
The parts $(1)$ and $(3)$ follow from Theorem \ref{pnQ2} and the part $(2)$ is a consequence of Theorem \ref{pnQp}.

\end{proof}

\end{thm}

We end the section with an example of application of Corollary \ref{numberfield}.

\begin{ex}
{\rm We bound the values of $w_4(\Z[\sqrt{2}])$ and $w_4(\Q(\sqrt{2}))$ from below. By Corollary \ref{numberfield} we need to consider $\mathfrak{p}$-adic completion of $\Z[\sqrt{2}]$ for each prime ideal $\mathfrak{p}$ of $\Z[\sqrt{2}]$. This completion is $\Z_p[\sqrt{2}]$, where $p$ is a unique prime number in $\mathfrak{p}$. Hence, we consider $\Z_p[\sqrt{2}]$ and $\Q_p(\sqrt{2})$, where $p$ runs through all the prime numbers.

If $p$ is an odd prime number, then the residue field of $\Z_p[\sqrt{2}]$ is $\FF_p$ for $p\equiv\pm 1\pmod{8}$ (in this case $\Z_p[\sqrt{2}]=\Z_p$) and $\FF_{p^2}$ for $p\equiv\pm 3\pmod{8}$. Since $p\nmid 4$, by Corollary \ref{pnR=...} and Theorem \ref{pnQR} it suffices to compute $s_4(\FF_{p^j})$ and $w_4(\FF_{p^j})$, $j\in\{1,2\}$, in order to obtain $w_4(\Z_p[\sqrt{2}])$ and $w_4(\Q_p(\sqrt{2}))$, respectively. If $p\geq 5$, then $\frac{p^2-1}{4}>p$, hence the subfield of sums of fourth powers in $\FF_{p^2}$ is $\FF_{p^2}$. In this case we can use \cite[inequality (5)]{winterhof} to obtain the inequality $w_4(\FF_{p^j})\leq 4$. As a result, $w_4(\Z_p[\sqrt{2}])$ and $w_4(\Q_p(\sqrt{2}))$ are less than or equal to $5$ for $p\geq 5$. If $p=3$, then the subfield of fourth powers in $\FF_9$ is $\FF_3$ and $w_4(\FF_9)=s_4(\FF_9)=1$, hence $w_4(\Z_3[\sqrt{2}])=w_4(\Q_3(\sqrt{2}))=2$, where every element of $\Z_3[\sqrt{2}]$ whose reduction modulo $3$ does not belong to $\FF_3$ is not a sum of fourth powers in $\Z_3[\sqrt{2}]$.

Now we consider the case of $p=2$. Then $\Z_2[\sqrt{2}]$ is a complete DVR with respect to the $\sqrt{2}$-adic valuation $\nu$, given by $\nu(a+b\sqrt{2})=\min\{2\nu_2(a),2\nu_2(b)+1\}$, $a,b\in\Z_2$. The valuation group of $\nu$ is $\Z$. Hence, applying Corollaries \ref{pnHDVR} and \ref{hDVR} we know that $w_4(\Z_2[\sqrt{2}])=w_4(\Z[\sqrt{2}]/(16\sqrt{2}))$ and $w_4(\Q_2(\sqrt{2}))=s_4(\Z[\sqrt{2}]/(16\sqrt{2}))+1$. Therefore we check the values of $(a+b\sqrt{2})^4\pmod{16\sqrt{2}}$ for $a,b\in\Z$. The table below shows all the possible values of $(a+b\sqrt{2})^4\pmod{16\sqrt{2}}$ depending on the values of $a$ and $b$ modulo $4$.

\bigskip
\begin{tabular}{|c|c|c|c|c|}
\hline
\diagbox{b \text{(mod 4)}}{a \text{(mod 4)}} & 0 & 1 & 2 & 3 \\
 \hline
 0 & 0 & 1 \text{or} 17 & 16 & 1 \text{or} 17\\
 \hline
1 & 4 & (1 \text{or} 17)+12$\sqrt{2}$ & 4 & (1 \text{or} 17)+4$\sqrt{2}$ \\
 \hline
 2 & 0 & (1 \text{or} 17)+8$\sqrt{2}$ & 16 & (1 \text{or} 17)+8$\sqrt{2}$\\
 \hline
 3 & 4 & (1 \text{or} 17)+4$\sqrt{2}$ & 4 & (1 \text{or} 17)+12$\sqrt{2}$ \\
 \hline
\end{tabular}

\bigskip
From the above table we see that 
$$w_4(\Z[\sqrt{2}]/(16\sqrt{2}))=7,\ s_4(\Z[\sqrt{2}]/(16\sqrt{2}))=6.$$
Hence,
\begin{align*}
    w_4(\Z_2[\sqrt{2}])=w_4(\Q_2(\sqrt{2}))=7.
\end{align*}

Summing up, we obtain the following lower bounds:
\begin{align*}
    w_4(\Z[\sqrt{2}])&\geq 7,\\
    w_4(\Q(\sqrt{2}))&\geq 7.
\end{align*}
}
\end{ex}

\section{Questions}
We end this paper with some open problems concerning sums of $n$th powers.
\renewcommand{\thefootnote}{$\dagger$} 

 Let $R$ be a valuation ring and $K$ its field of fractions.
\begin{prob}
Determine conditions on $R$ for which we have
$$w_n(R)=w_n(K).$$
\end{prob}
To be precise, we ask for determining real or complex numbers such that the existence of their embeddings in $R$ implies the equality $w_n(R)=w_n(K)$, without computing $w_n(R)$ or $w_n(K)$. In case of $n=2$, which motivates the above problem, \cite[Theorem 4.5]{CLRR} proved by Kneser and Colliot-Th\'el\`ene states that $\tfrac12 \in R$ implies $w_2(R)=w_2(K)$. For $n=3$ one has to expect that some nonrational numbers have to be involved as the following example of complete discrete valuation ring of equicharacteristic zero shows. Obviously, $s_3(\Q)=1$ and $w_3(\Q)=3$ by Theorem of Ryley (cf. \cite{richmond}). Theorem \ref{Main Theorem} implies that $w_3(\Q((x)))=2<w_3(\Q[[x]])=3$ meanwhile $w_3(\R((x)))=w_3(\R[[x]])=2$.

 The second problem deals with a general type of finitely generated affine $K$-algebras.
 \begin{prob}
 Let $I$ be an ideal of the ring $K[x_1,x_2,\dots, x_s]$. Compute
 $$w_n(K[x_1,x_2,\dots, x_s]/I).$$
 \end{prob}
Obviously, neither Theorem \ref{Main Theorem} nor Corollary \ref{algset2} can be applied to every such an algebra. In general, this can be very difficult task. We present a simple example which helps us to visualize the difficulties.

\begin{ex}{\rm
Let $K$ be a field with $s_2(K)\geq 2$. Consider the $K$-algebra $R=K[x,y,z]/(x^2+y^2)$. Then, the $2$nd Waring number of $R$ is equal to $s_2(K)+1$. This follows from the surjection $R\rightarrow K[z]$. However, it is not clear, what is the value of  $w_2(K[x,y]/(x^2+y^2))$. On the other hand, Waring numbers can significantly drop when passing to the field of rational functions. Since $s_2(\mathrm{Frac}(K[x,y]/(x^2+y^2)))=1$, the $2$nd Waring number of this field is at most $2$.}
\end{ex}

\section*{Acknowledgements}
The authors wish to thank Pavlo Yatsyna for his suggestion concerning the nomenclature used in the paper. Also, express their gratitude to the anonymous referee for useful comments that greatly improved edition and literature.

\vspace{5pt}
\begin{small}
\noindent
Tomasz Kowalczyk

\noindent
Institute of Mathematics

\noindent
Faculty of Mathematics and Computer Science

\noindent
Jagiellonian University

\noindent
ul. Łojasiewicza 6, 30-348 Kraków, Poland

\noindent
e-mail: tomek.kowalczyk@uj.edu.pl

\vspace{5pt}
\noindent
Piotr Miska

\noindent
Institute of Mathematics

\noindent
Faculty of Mathematics and Computer Science

\noindent
Jagiellonian University

\noindent
ul. Łojasiewicza 6, 30-348 Kraków, Poland

\noindent
e-mail: piotr.miska@uj.edu.pl

\end{small}	
\end{document}